\documentclass[12pt]{amsart}
\usepackage{amsmath, amssymb}
\usepackage[T2A]{fontenc}
\usepackage{geometry}
\usepackage{amsaddr}

\usepackage{color,graphicx}

\usepackage{url}
\usepackage[colorlinks = true,
linkcolor=blue,
urlcolor=blue,
citecolor=blue,
anchorcolor=blue,
backref=page]{hyperref}
\usepackage{tikz}

\newtheorem{theorem}{Theorem}[section]
\newtheorem{lemma}{Lemma}[section]

\theoremstyle{definition}

\newtheorem{remark}{Remark}[section]

\begin{document}
	
	\author{A.\,A.~Egorov}
	\address{Sobolev Institute of Mathematics, Siberian Branch of Russian Academy of Sciences, Novosibirsk,	Russia} 
	\email{a.egorov2@g.nsu.ru}
	
	\title[Lower bounds for right-angled hyperbolic polyhedra]{On lower bounds for the number of ideal and finite vertices of right-angled hyperbolic polyhedra in dimensions from 5 to 12}
	
	\maketitle 
	
	\begin{abstract}
		We investigate lower bounds for the number of ideal and finite vertices of right-angled hyperbolic polyhedra of finite volume. We use a geometric method of orthogonal gluings to establish new bounds in low dimensions, specifically  $v_\infty(P^5) \ge 3$ and $v_{fin}(P^7) \ge 4$. By combining these initial bounds with double counting arguments and recurrence relations, we obtain improved lower bounds for both types of vertices in all higher dimensions up to $n=12$, the maximal dimension where polyhedra of this class exist.
	\end{abstract}
	
	
	\let\thefootnote\relax\footnote{The author was supported by the state task of Sobolev Institute of Mathematics, project No. FWNF-2026-0011.}
	
	\section{Introduction}
	
	Classification and study of geometric and combinatorial characteristics of convex polyhedra in Lobachevsky spaces $\mathbb{H}^n$ is a classical and actively developing problem of hyperbolic geometry. A convex polyhedron is called right-angled if all its dihedral angles are equal to $\pi/2$. Unlike spherical and Euclidean geometries, where only parallelepipeds (or simplexes in the spherical case) can be right-angled, hyperbolic geometry allows the existence of a rich variety of right-angled polyhedra, but only in low dimensions. In dimension $n=3$, Andreev's theorem~\cite{And70,And70_2} gives a complete combinatorial description of right-angled polyhedra of finite volume, showing that their topological diversity is infinite. In a recent paper~\cite{EV2025} it is shown that among all three-dimensional hyperbolic polyhedra, a triangular bipyramid with two finite and three ideal vertices has the minimal volume.
	
	With the growth of dimension $n$, the combinatorial structure of right-angled polyhedra becomes increasingly rigid. In the work of E.\,B.~Vinberg~\cite{Vin84} it was proved that compact Coxeter polyhedra -- polyhedra with angles of the form $\tfrac{\pi}{k_i}$, where $k_i$ are integers, do not exist in dimensions $n \ge 30$. At the same time, examples of such polyhedra are known only in dimensions no more than 8. Later, A.\,G.~Khovanskii~\cite{Kho86} and M.\,N.~Prokhorov~\cite{Pr86} showed that Coxeter polyhedra of finite volume do not exist in dimensions $n \ge 996$. At the same time, examples of such polyhedra are known only in dimensions no higher than $n = 21$.
	
	For a narrower class of right-angled polyhedra, the dimension bounds are significantly stricter. L.~Potyagailo and E.\,B.~Vinberg~\cite{PV05} proved that compact right-angled polyhedra do not exist for $n \ge 5$, and right-angled polyhedra of finite volume for $n \geq 15$. Subsequently, G.~Dufour~\cite{Dufour2010}, developing ideas from the work of Potyagailo and Vinberg, improved the bound for polyhedra of finite volume, proving their absence for $n \geq 13$. We also note the result of A.~Kolpakov~\cite{Ko2012}, who established the absence of ideal right-angled polyhedra for $n \geq 7$. Examples of right-angled polyhedra of finite volume are known only in dimensions no higher than $n = 8$.
	
	The key structural parameters determining the possibility of the existence of right-angled polyhedra of finite volume in higher dimensions are the numbers of their ideal and finite vertices. Lower bounds on these quantities are closely interconnected: an improvement of the lower bound for the number of ideal vertices in dimension $n$ entails an improvement of the lower bound for the number of facets in dimension $n+1$ via Dufour's recurrence relations~\cite{Dufour2010}, and this, in turn, improves the lower bounds on the number of finite vertices via the double counting lemma. Due to the multiplicative nature of the double counting lemmas, even a small strengthening of the bound in the base dimension generates significant nonlinear growth of the final bounds upon transition to higher dimensions.
	
	In the present work, we propose a unified method of orthogonal gluings, allowing us to solve this problem and obtain new base bounds simultaneously for both types of vertices. The main results of the article are the following two theorems.
	
	\begin{theorem} \label{thm:main_ideal_intro}
		Let $P \subset \mathbb{H}^5$ be a right-angled hyperbolic polyhedron of finite volume. Then the number of its ideal vertices $v_\infty(P) \ge 3$. 
	\end{theorem}
	
	\begin{theorem} \label{thm:main_fin_intro}
		Let $P \subset \mathbb{H}^7$ be a right-angled hyperbolic polyhedron of finite volume. Then the number of its finite vertices $v_{fin}(P) \ge 4$.
	\end{theorem}
	
	The first theorem improves Alexandrov's bound $v_\infty(P^5) \geq 2$~\cite{Ale23}. The second theorem, relying in its proof on Kolpakov's theorem, establishes a new strict base for finite vertices. Both results launch a recurrent process that gives significantly improved lower bounds for all dimensions $6 \leq n \leq 12$, presented in Tables~\ref{tab:comparison} and~\ref{tab:comparison_fin}.
	
	The structure of the present article is as follows. Section 2 provides the necessary preliminary information about right-angled polyhedra and formulates lemmas on face combinatorics, upon which the further proof relies. Section 3 contains the proof of Theorem~\ref{thm:main_ideal_intro} for the 5-dimensional space and the calculation of recurrent bounds for higher dimensions. In Section 4, Theorem~\ref{thm:main_fin_intro} about the minimal number of finite vertices for dimension 7 is proved, and a cascade of new restrictions on $v_{fin}(P^n)$ for $n \ge 8$ is also calculated.
	
	\section{Preliminaries}
	
	Let $\mathbb{H}^n$ be the $n$-dimensional hyperbolic space, and $\partial\mathbb{H}^n$ be its boundary (absolute). A convex polyhedron $P \subset \mathbb{H}^n$ is the intersection of a finite number of closed hyperbolic half-spaces, having a non-empty interior. In this work, we consider only polyhedra of finite volume. The vertices of the polyhedron $P$ lying in $\mathbb{H}^n$ are called finite, and the vertices lying on the absolute $\partial\mathbb{H}^n$ are called ideal. We denote the number of ideal vertices of the polyhedron $P$ by $v_\infty(P)$.
	
	A face of dimension $k$ (or $k$-dimensional face) of a polyhedron $P$ is a non-empty intersection of $P$ with a set of hyperplanes bounding these half-spaces, if the dimension of this intersection is equal to $k$. Faces of codimension 1 (that is, of dimension $n-1$) are called facets.
	
	A polyhedron $P$ is called right-angled if the dihedral angle between any two intersecting facets is equal to $\pi/2$. The class of all right-angled hyperbolic polyhedra of finite volume in $\mathbb{H}^n$, having at least one ideal vertex, will be denoted by $\mathcal{P}^n$. For $P \in \mathcal{P}^n$, we denote by $a_k(P)$ the number of $k$-dimensional faces, where $0 \le k \le n-1$. In particular, $a_{n-1}(P)$ is the number of facets.
	
	The local structure of a polyhedron $P \in \mathcal{P}^n$ in a neighborhood of a finite vertex of an $n$-dimensional right-angled polyhedron $P$ is a cone over an $(n-1)$-dimensional simplex, and in a neighborhood of an ideal vertex is a cone over an $(n-1)$-dimensional parallelepiped (see~\cite{Dufour2010}). As a consequence, a right-angled polyhedron $P$ (as, in fact, any acute-angled one -- see~\cite{AVS}) is simple at edges (each edge belongs to exactly $n-1$ facets). This at the same time means that any $k$-face $F^k$ is the intersection of exactly $n-k$ facets and is not contained in any other facets of the polyhedron (see~\cite{AVS}). Moreover, this means that exactly $n$ facets converge at each finite vertex of a right-angled $n$-polyhedron, and $2(n-1)$ facets at each ideal one.
	
	Recall the classical inequalities of V.\,V.~Nikulin, generalized by A.\,G.~Khovanskii to the case of polyhedra simple at edges.
	
	\begin{theorem}[Nikulin--Khovanskii inequality, \cite{Nik81, Kho86}] \label{thm:nikulin}
		Let $P$ be an $n$-dimensional polyhedron, simple at edges. Denote by $a_k^l(P)$ the average number of $l$-dimensional faces contained in one $k$-dimensional face of the polyhedron $P$, where $0 \leq l < k \leq \lceil n/2 \rceil$. Then
		\begin{equation}
			a_k^l(P) < \binom{n-l}{n-k} \frac{\binom{\lfloor n/2 \rfloor}{l} + \binom{\lceil n/2 \rceil}{l}}{\binom{\lfloor n/2 \rfloor}{k} + \binom{\lceil n/2 \rceil}{k}}.
		\end{equation}
	\end{theorem}
	
	Nonaka investigated three-dimensional right-angled hyperbolic polyhedra with one ideal vertex and obtained the following result.
	
	\begin{lemma}[Nonaka, \cite{Nonaka}] \label{lem:nonaka_bound}
		If $P^3$ is a right-angled hyperbolic 3-polyhedron of finite volume, and the number of its ideal vertices satisfies the condition $v_\infty(P^3) \le 1$. Then the number of its two-dimensional faces $a_2(P^3)$ is bounded from below as $a_2(F^3) \ge 12$.
	\end{lemma}
	
	From Theorem~\ref{thm:nikulin} for dimension $n=5$ ($l=2, k=3$) follows a strict upper bound $\bar{a}_2 = a_3^2(P^5) < 12$. A consequence of this restriction and Nonaka's result is a result from the preprint by S.\,A.~Alexandrov. Below, for completeness of presentation, its proof is also reproduced.
	
	\begin{lemma}[Alexandrov, \cite{Ale23}] \label{lem:alex_base}
		For any polyhedron $P^5 \in \mathcal{P}^5$, the inequality $v_\infty(P^5) \ge 2$ holds.
	\end{lemma}
	\begin{proof}
		Assume the contrary: let $v_\infty(P^5) \le 1$. In this case, each 3-dimensional face $F^3$ of the polyhedron $P^5$ is a right-angled hyperbolic 3-polyhedron of finite volume, having no more than one ideal vertex. According to Nonaka's lemma~\ref{lem:nonaka_bound}, each such 3-face must contain at least 12 two-dimensional faces. Consequently, the average number of two-dimensional faces in the 3-faces of $P^5$ is bounded from below as $a_3^2(P^5) \ge 12$. On the other hand, according to Nikulin--Khovanskii inequality (Theorem~\ref{thm:nikulin}), for any 5-dimensional polyhedron, simple at edges, the reverse relation must hold: $a_3^2(P^5) < 12$. The obtained contradiction proves that $v_\infty(P^5) \ge 2$.
	\end{proof}
	
	For the inductive extension of bounds to higher dimensions, we will need the following constants. Denote by $\nu_n$ the minimum possible value of the sum of the number of facets and the number of ideal vertices for an $n$-dimensional polyhedron: $$\nu_n = \min_{P \in \mathcal{P}^n} (a_{n-1}(P) + v_\infty(P)).$$ In Dufour's work, the following bounds are obtained:
	
	\begin{lemma}[Dufour, \cite{Dufour2010}] \label{lem:dufour_recurrence}
		For any $n \ge 3$ and $P^n \in \mathcal{P}^n$, the inequalities hold:
		\begin{equation}
			a_{n-1}(P^n) \ge 1 + \nu_{n-1}, \quad \text{and} \quad \nu_n \ge 5 - 2n + 2\nu_{n-1}.
		\end{equation}
		Exact values and bounds of the initial constants are known: $\nu_3 = 9$, $\nu_4 = 15$, $\nu_5 \ge 26$, $\nu_6 \ge 45$, $\nu_7 \ge 81$.
	\end{lemma}
	
	Finally, the recurrent transition to higher dimensions for the number of ideal vertices is carried out on the basis of the following lemma. Below, for completeness of presentation, its proof is also reproduced.
	
	\begin{lemma}[Alexandrov, \cite{Ale23}] \label{lem:alex_recurrence}
		Let $P^n \in \mathcal{P}^n$. Then the minimal number of its ideal vertices is bounded from below by the relation:
		$$
		v_\infty(P^n) \ge \left\lceil \frac{a_{n-1}(P^n) \cdot v'_\infty(P^n)}{2(n-1)} \right\rceil,
		$$
		where $v'_\infty(P^n)$ is the minimal number of ideal vertices that a facet of the polyhedron $P^n$ can contain.
	\end{lemma}
	\begin{proof}
		Since each ideal vertex of the polyhedron $P^n$ belongs to exactly $2(n-1)$ facets, the following inequality holds
		$$
		v_\infty(P^n) \cdot 2(n-1) = \sum_{\dim F = n-1} v_\infty(F) \ge a_{n-1}(P^n) \cdot v'_\infty(P^n),
		$$
		where the summation is taken over all facets $F$ of the polyhedron $P^n$. Expressing $v_\infty(P^n)$ and taking into account that the number of vertices must be an integer, we arrive at the required inequality.
	\end{proof}
	
	\section{Bounds for ideal vertices}
	
	\subsection{Bound for 5-dimensional polyhedra}
	
	In this section we will prove that the case $v_\infty(P^5) = 2$ is geometrically impossible.
	
	\begin{theorem} \label{thm:dim5}
		Let $P \subset \mathbb{H}^5$ be a right-angled hyperbolic polyhedron of finite volume. Then the number of its ideal vertices is strictly greater than two, that is, $v_\infty(P) \ge 3$.
	\end{theorem}
	
	\begin{proof}
		According to the base bound from Lemma \ref{lem:alex_base}, we have the inequality $v_\infty(P) \ge 2$. It remains to prove that the case $v_\infty(P) = 2$ is impossible. Assume the contrary: let the polyhedron $P$ have exactly two ideal vertices, which we will denote by $u$ and $v$. The intersection of any non-empty set of faces of a convex polyhedron, if it is not empty, is itself a face. Since the vertices $u$ and $v$ belong to $P$, the set of faces containing both of them is not empty (at a minimum, the polyhedron $P$ itself is such a face). Consequently, there exists a unique inclusion-minimal common face $F_{min} \subset P$, simultaneously containing $u$ and $v$. Denote its dimension by $d = \dim F_{min}$. Possible values of $d$ are 1, 2, 3, or a minimal common face of dimension $d \le 3$ does not exist.
		
		Let us describe the reduction algorithm. We show that if $d \le 3$, there exists an operation allowing one to construct a new right-angled polyhedron $P'$, for which $v_\infty(P') = 2$, but the dimension of the minimal common face $d'$ is strictly greater than $d$.
		
		\textbf{Operation of orthogonal gluing.} Let $H_1$ be a facet of the polyhedron $P$ containing $F_{min}$ (and therefore, containing $u$ and $v$). Reflect $P$ with respect to the hyperplane containing $H_1$, and obtain the polyhedron $\rho_{H_1}(P)$. Consider the union $P' = P \cup_{H_1} \rho_{H_1}(P)$.
		
		The new polyhedron $P'$ is a right-angled polyhedron of finite volume. Indeed, any facet $H_i \subset P$ ($i \neq 1$) having a non-empty intersection with $H_1$ is orthogonal to it due to the right-angledness of $P$. Consequently, the hyperplane containing $H_i$ is invariant under the reflection $\rho_{H_1}$, and the union $H'_i = H_i \cup \rho_{H_1}(H_i)$ forms a single, smooth (totally geodesic) facet in the polyhedron $P'$. The dihedral angles between the intersecting facets $H'_i$ and $H'_j$ remain right.
		
		The facet $H_1$ itself in the polyhedron $P'$ ceases to be an element of the boundary: each of its points acquires a full neighborhood, composed of semi-neighborhoods in $P$ and $\rho_{H_1}(P)$, and becomes an interior point of $P'$. Thus, $H_1$ is excluded from the composition of facets of the new polyhedron.
		
		Since the ideal vertices $u, v \in H_1$, they are invariant under reflection and are identified with their copies. Since $P$ does not contain other ideal vertices, new cusps do not arise, and $v_\infty(P') = 2$.
		
		\textbf{Reduction step 1 ($d=1 \to d=2$).} 
		Let the minimal common face for $u$ and $v$ be a 1-dimensional edge $e$. In a 5-dimensional space, an edge of an acute-angled polyhedron is the intersection of exactly $5-1 = 4$ facets: $e = H_1 \cap H_2 \cap H_3 \cap H_4$. This means that the vertices $u$ and $v$ jointly belong to exactly these four facets. We perform the gluing operation $P' = P \cup_{H_1} \rho_{H_1}(P)$ along the facet $H_1$. In the new right-angled polyhedron $P'$, the set $H_1$ has gone to the interior. Consequently, the list of facets of the polyhedron $P'$, simultaneously containing $u$ and $v$, is reduced to three: $H'_2, H'_3, H'_4$. The intersection of these three facets determines a new minimal face of dimension $5-3 = 2$. Thus, in $P'$ the vertices $u$ and $v$ are no longer connected by an edge, and the dimension of their minimal common face increases to $d'=2$.
		
		\textbf{Reduction step 2 ($d=2 \to d=3$).} 
		Let $F_{min} = F^2$ ($d=2$). A two-dimensional face is defined by the intersection of exactly $5-2 = 3$ facets: $F^2 = H_1 \cap H_2 \cap H_3$. The vertices $u$ and $v$ jointly belong only to them. We perform a gluing of $P'$ along the facet $H_1$. The set $H_1$ becomes interior. In the polyhedron $P'$, the ideal vertices $u$ and $v$ jointly belong to exactly two facets: $H'_2$ and $H'_3$. The intersection $H'_2 \cap H'_3$ defines a face of dimension $5-2 = 3$. Consequently, the dimension of the minimal common face in $P'$ strictly increases to $d'=3$.
		
		\textbf{Reduction step 3 ($d=3 \to d \geq 4$).} 
		Let $F_{min} = F^3$ ($d=3$). This face is the intersection of exactly $5-3=2$ facets: $F^3 = H_1 \cap H_2$. The ideal vertices $u$ and $v$ jointly belong exclusively to these two facets. We perform a gluing of $P'$ along the facet $H_1$. In the resulting polyhedron $P'$, the set $H_1$ becomes interior, and the vertices $u, v$ jointly belong to exactly one facet $H'_2$. Since they no longer belong jointly to any two facets, they do not lie in any common face of dimension 3 or less on the boundary of $P'$.
		
		\textbf{Completion of the proof.} 
		Applying the described reduction algorithm (having performed from zero to three gluing steps), we in any case arrive at a right-angled hyperbolic polyhedron of finite volume $\tilde{P}$, for which $v_\infty(\tilde{P}) = 2$, but the ideal vertices do not lie simultaneously in any common face of dimension 3 or less. Consequently, any 3-face $F^3 \subset \tilde{P}$ satisfies the condition $v_\infty(F^3) \le 1$. According to Lemma \ref{lem:nonaka_bound}, each right-angled hyperbolic 3-face with such a number of ideal vertices must have a number of two-dimensional faces $a_2(F^3) \ge 12$. Hence, the average number of 2-faces in the 3-faces of the polyhedron $\tilde{P}$ is strictly bounded from below:
		\begin{equation} \label{eq:a2_lower}
			\bar{a}_2 \ge 12.
		\end{equation}
		On the other hand, the application of the Nikulin--Khovanskii inequality (Theorem \ref{thm:nikulin}) to the polyhedron $\tilde{P}$ gives a strict upper bound on this same quantity:
		\begin{equation} \label{eq:a2_upper}
			\bar{a}_2 < 12.
		\end{equation}
		A comparison of inequalities \eqref{eq:a2_lower} and \eqref{eq:a2_upper} leads to a contradiction. Consequently, the initial assumption that $v_\infty(P) = 2$ is false, and $v_\infty(P) \ge 3$. The theorem is proved.
	\end{proof}
	
	\begin{remark}
		Note that potentially the method of sequential orthogonal gluings can be modified to exclude 5-dimensional polyhedra with three ideal vertices, which would give the bound $v_\infty(P) \ge 4$. However, for $v_\infty(P) \ge 5$, this approach apparently encounters an insurmountable obstacle -- a polyhedron with four ideal vertices can contain two 3-faces that do not lie in a common 4-face and each of which contains two ideal vertices.
	\end{remark}
	
	\subsection{Bounds for ideal vertices for higher dimensions}
	
	The new strict base $v_5 \ge 3$ established in Theorem \ref{thm:dim5} allows launching a recurrent process for calculating the minimal number of ideal vertices in spaces $\mathbb{H}^n$ for $n \ge 6$. We denote by $v_n = \min_{P \in \mathcal{P}^n} v_\infty(P)$ the exact lower bound of the number of ideal vertices for the $n$-dimensional space.
	
	\begin{theorem} \label{thm:higher_dims}
		The minimal number of ideal vertices $v_n$ of a right-angled hyperbolic polyhedron of finite volume in dimensions $6 \le n \le 12$ is strictly greater than previously known bounds. Exact values of the new lower bounds are given in Table~\ref{tab:comparison}.
	\end{theorem}
	
	\begin{proof}
		We apply the recurrence formula $v_n \ge \left\lceil \frac{a_{n-1}(P^n) \cdot v_{n-1}}{2(n-1)} \right\rceil$ (Lemma \ref{lem:alex_recurrence}). To estimate the minimal number of facets $a_{n-1}(P^n)$, we use the inequality $a_{n-1}(P^n) \ge 1 + \nu_{n-1}$ (Lemma \ref{lem:dufour_recurrence}). The constant $\nu_n$, defined as the minimum possible sum of the number of facets and ideal vertices, is updated at each step according to the rule $\nu_n \ge \max(5 - 2n + 2\nu_{n-1}, \, a_{n-1}(P^n) + v_n)$.
		
		\textbf{Dimension $n=6$:} It is known that $\nu_5 \ge 26$, consequently, the minimal number of facets of a 6-dimensional polyhedron $a_5(P^6) \ge 1 + 26 = 27$. Substituting the new base $v_5 \ge 3$, we obtain:
		$$v_6 \ge \left\lceil \frac{27 \cdot 3}{2 \cdot 5} \right\rceil = \lceil 8.1 \rceil = 9.$$
		Let us recalculate the constant $\nu_6$. According to the linear bound of Lemma \ref{lem:dufour_recurrence}: $\nu_6 \ge 5 - 12 + 2(26) = 45$. On the other hand, by the sum $a_5(P^6) + v_6 \ge 27 + 9 = 36$. We take the maximum: $\nu_6 \ge 45$.
		
		\textbf{Dimension $n=7$:} Using $\nu_6 \ge 45$, we find $a_6(P^7) \ge 1 + 45 = 46$. We calculate the number of ideal vertices:
		$$v_7 \ge \left\lceil \frac{46 \cdot 9}{12} \right\rceil = \lceil 34.5 \rceil = 35.$$
		We update $\nu_7$: the linear bound gives $5 - 14 + 2(45) = 81$. The bound by sum gives $46 + 35 = 81$. Total $\nu_7 \ge 81$.
		
		\textbf{Dimension $n=8$:} From $\nu_7 \ge 81$ it follows that $a_7(P^8) \ge 1 + 81 = 82$. Then:
		$$v_8 \ge \left\lceil \frac{82 \cdot 35}{14} \right\rceil = \left\lceil \frac{2870}{14} \right\rceil = 205.$$
		Recalculation of $\nu_8$: the linear bound gives $5 - 16 + 2(81) = 151$. The direct sum gives $82 + 205 = 287$. The maximum is equal to $\nu_8 \ge 287$.
		
		Starting from this step, the direct sum of the number of faces and vertices begins to dominate.
		
		\textbf{Dimension $n=9$:} We use $\nu_8 \ge 287$, whence $a_8(P^9) \ge 1 + 287 = 288$.
		$$v_9 \ge \left\lceil \frac{288 \cdot 205}{16} \right\rceil = \left\lceil \frac{59040}{16} \right\rceil = 3690.$$
		The constant $\nu_9 \ge 288 + 3690 = 3978$.
		
		\textbf{Dimension $n=10$:} We have $a_9(P^{10}) \ge 1 + 3978 = 3979$. We calculate the vertices:
		$$v_{10} \ge \left\lceil \frac{3979 \cdot 3690}{18} \right\rceil = \left\lceil \frac{14\,682\,510}{18} \right\rceil = 815\,695.$$
		The constant $\nu_{10} \ge 3979 + 815\,695 = 819\,674$.
		
		\textbf{Dimension $n=11$:} We obtain $a_{10}(P^{11}) \ge 1 + 819\,674 = 819\,675$.
		$$v_{11} \ge \left\lceil \frac{819\,675 \cdot 815\,695}{20} \right\rceil = \left\lceil \frac{668\,604\,799\,125}{20} \right\rceil = 33\,430\,239\,957.$$
		The constant $\nu_{11} \ge 819\,675 + 33\,430\,239\,957 = 33\,431\,059\,632$.
		
		\textbf{Dimension $n=12$:} For the limit dimension $n=12$, the number of facets $a_{11}(P^{12}) \ge 1 + 33\,431\,059\,632 = 33\,431\,059\,633$.
		$$v_{12} \ge \left\lceil \frac{33\,431\,059\,633 \cdot 33\,430\,239\,957}{22} \right\rceil = \left\lceil \frac{1\,117\,608\,345\,547\,966\,355\,781}{22} \right\rceil = 50\,800\,381\,957\,715\,834\,354.$$
		
		The performed calculations clearly demonstrate the combinatorial explosion of the required number of ideal vertices for large dimensions.
	\end{proof}
	
	\begin{table}[h]
		\centering
		\renewcommand{\arraystretch}{1.4}
		\begin{tabular}{|c|c|c|c|}
			\hline
			Dimension & Nonaka \cite{Nonaka} & Estimate from \cite{Ale23} & \textbf{New estimate} \\
			$n$ & $v_\infty(P^n) \ge$ & $v_\infty(P^n) \ge$ & $\mathbf{v_\infty(P^n) \ge}$ \\
			\hline
			5 & 1 & 2 & \textbf{3} \\
			6 & 3 & 6 & \textbf{9} \\
			7 & 17 & 23 & \textbf{35} \\
			8 & 36 & 135 & \textbf{205} \\
			9 & 91 & 1 704 & \textbf{3 690} \\
			10 & 254 & 182 044 & \textbf{815 695} \\
			11 & 741 & $1.67 \cdot 10^9$ & \textbf{33 430 239 957} \\
			12 & 2200 & $1.27 \cdot 10^{17}$ & \textbf{50 800 381 957 715 834 354} \\
			\hline
		\end{tabular}
		\caption{Improved exact lower bounds on the number of ideal vertices $v_\infty(P^n)$.}
		\label{tab:comparison}
	\end{table}
	
	\section{Bounds for finite vertices}
	
	In this section, we apply the method of orthogonal gluings to obtain lower bounds for the number of finite vertices. Unlike the previous section, where the method was used to exclude configurations with a small number of ideal vertices, here we sequentially exclude the possibility of the existence of a polyhedron with a small number of finite vertices, relying on Kolpakov's theorem and parity properties.
	
	\subsection{Bound for 7-dimensional polyhedra}
	
	\begin{theorem} \label{thm:dim7_fin}
		Let $P \subset \mathbb{H}^7$ be a right-angled hyperbolic polyhedron of finite volume. Then the number of its finite vertices $v_{fin}(P) \ge 4$.
	\end{theorem}
	
	\begin{proof}
		In any right-angled 3-dimensional polyhedron, the degree of each finite vertex is 3, and of each ideal one is 4 (see \cite{Dufour2010}). From the sum of degrees formula $2E = 3v_{fin} + 4v_\infty$, it follows that the value $3v_{fin}$ must be even, and therefore, the number of finite vertices $v_{fin}$ itself is always even. Consequently, any 3-dimensional right-angled face cannot have an odd number of finite vertices.
		
		According to Kolpakov's theorem~\cite{Ko2012}, ideal right-angled polyhedra of finite volume in dimensions $n \geq 7$ do not exist. From this it follows that the polyhedron $P^7$ has at least one finite vertex, that is, $v_{fin}(P^7) \ge 1$.
		
		Let us prove that $v_{fin}(P) \ge 4$. To do this, we sequentially exclude the cases with 2 and 3 vertices.
		
		\textbf{Case 1: $v_{fin}(P^7) \ne 1$.} 
		Assume the contrary: let $v_{fin}(P^7) = 1$, then there exists a 3-face of the polyhedron $P^7$ containing exactly 1 finite vertex, which is impossible.
		
		\textbf{Case 2: $v_{fin}(P^7) \ne 2$.} 
		Let $v_{fin}(P^7) = 2$, we denote the two only finite vertices by $u$ and $v$. Take any 3-dimensional face $F^3$ containing vertex $u$. Since $F^3$ itself is a right-angled 3-polyhedron, it must have an even number of finite vertices. Since there are only two of them in the entire 7-dimensional polyhedron, the face $F^3$ must also contain vertex $v$. Consequently, vertices $u$ and $v$ lie in one common 3-dimensional face.
		
		A three-dimensional face in a 7-dimensional polyhedron, simple at edges, is the intersection of exactly $7-3=4$ facets. This means that the vertices $u$ and $v$ jointly lie in these 4 facets. Let $H_1$ be one of them. We perform the operation of orthogonal gluing $P' = P \cup_{H_1} \rho_{H_1}(P)$ along the facet $H_1$, similar to the one used in the proof of Theorem \ref{thm:dim5}. Since the vertices $u$ and $v$ lie on $H_1$, under reflection they transition into themselves.
		
		In a right-angled polyhedron, each $k$-dimensional face (except an ideal vertex) is defined by the intersection of exactly $n-k$ facets. In the initial polyhedron $P$, the finite vertices (0-dimensional faces) were the intersection point of exactly 7 orthogonal facets. In the new polyhedron $P'$, the facet $H_1$ becomes interior. Consequently, the points $u$ and $v$ remain in the intersection of only 6 orthogonal facets forming the boundary of $P'$. In a 7-dimensional space, the intersection of 6 facets forms a 1-dimensional edge.
		
		Thus, the points $u$ and $v$ in the polyhedron $P'$ become interior points of edges and cease to be vertices. Since there were no other finite vertices in $P$, in the polyhedron $P'$ no finite vertices remain at all, that is, $v_{fin}(P') = 0$. We obtained $P'$ — an ideal right-angled polyhedron in $\mathbb{H}^7$, which directly contradicts Kolpakov's theorem~\cite{Ko2012}. Consequently, $v_{fin}(P^7) \ne 2$.
		
		\textbf{Case 3: $v_{fin}(P^7) \ne 3$.}
		Assume that $v_{fin}(P^7) = 3$, and denote these finite vertices by $u, v, w$. Take any 3-dimensional face $F^3$ containing vertex $u$. Due to the parity requirement for the number of finite vertices in $F^3$, this face must contain exactly one of the remaining vertices, for example $v$. Hence, $u$ and $v$ again belong to a common 3-face and, as a consequence, to a common facet $H_1$. We perform the gluing $P' = P \cup_{H_1} \rho_{H_1}(P)$ along this facet $H_1$.
		
		Similar to Case 2, the vertices $u$ and $v$ will transition to interior points of edges and cease to be vertices of the polyhedron $P'$. Let us consider what will happen to the third vertex $w$. Two variants are possible:
		1) If $w \in H_1$, then it also transitions to an interior point of an edge. Then in $P'$ there will remain 0 finite vertices, which again contradicts Kolpakov's theorem.
		2) If $w \notin H_1$, then under reflection it is duplicated. In the polyhedron $P'$, exactly two finite vertices are formed: $w$ and its mirror copy $\rho_{H_1}(w)$. However, in Case 1 we already proved that a 7-dimensional right-angled polyhedron of finite volume cannot have exactly 2 finite vertices. We again arrive at a contradiction.
		
		Thus, configurations with 1, 2 and 3 finite vertices are impossible. The obtained contradictions finally prove that $v_{fin}(P^7) \ge 4$.
	\end{proof}
	
	\subsection{Bounds for finite vertices for higher dimensions}
	
	Similar to Lemma \ref{lem:alex_recurrence}, we formulate a lemma on the double counting of incidences for finite vertices.
	
	\begin{lemma} \label{lem:fin_recurrence}
		Let $P^n \in \mathcal{P}^n$. Then the minimal number of its finite vertices is bounded from below by the relation:
		$$v_{fin}(P^n) \ge \left\lceil \frac{a_{n-1}(P^n) \cdot v'_{fin}(P^n)}{n} \right\rceil$$
		where $v'_{fin}(P^n)$ is the minimal number of finite vertices in a facet of the polyhedron $P^n$.
	\end{lemma}
	\begin{proof}
		Summing the number of finite vertices over all facets, we obtain $v_{fin}(P^n) \cdot n = \sum_{\dim F = n-1} v_{fin}(F) \ge a_{n-1}(P^n) \cdot v'_{fin}(P^n)$. Expressing $v_{fin}(P^n)$ and taking into account the integrality of the number of vertices, we arrive at the required inequality.
	\end{proof}
	
	Applying the proven lemma to the new base $v_{fin}(P^7) \ge 4$ and using the minimal number of facets $a_{n-1}(P^n)$ calculated in section 3, we obtain a cascade of new lower bounds for higher dimensions.
	
	\begin{theorem} \label{thm:higher_dims_fin}
		The number of finite vertices of a right-angled hyperbolic polyhedron of finite volume in dimensions $8 \le n \le 12$ is bounded from below by the values given in Table~\ref{tab:comparison_fin}.
	\end{theorem}
	
	\begin{proof}
		Let us perform sequential calculations using the values $a_{n-1}(P^n)$ from the proof of Theorem \ref{thm:higher_dims}:
		
		\textbf{Dimension $n=8$:} Using $a_7(P^8) \ge 82$ and $v_{fin}(P^7) \ge 4$, we obtain: 
		$$v_{fin}(P^8) \ge \left\lceil \frac{82 \cdot 4}{8} \right\rceil = 41.$$
		
		\textbf{Dimension $n=9$:} Using $a_8(P^9) \ge 288$ and $v_{fin}(P^8) \ge 41$:
		$$v_{fin}(P^9) \ge \left\lceil \frac{288 \cdot 41}{9} \right\rceil = 1312.$$
		
		\textbf{Dimension $n=10$:} Using $a_9(P^{10}) \ge 3979$ and $v_{fin}(P^9) \ge 1312$:
		$$v_{fin}(P^{10}) \ge \left\lceil \frac{3979 \cdot 1312}{10} \right\rceil = \lceil 522044.8 \rceil = 522045.$$
		
		\textbf{Dimension $n=11$:} Using $a_{10}(P^{11}) \ge 819675$ and $v_{fin}(P^{10}) \ge 522045$:
		$$v_{fin}(P^{11}) \ge \left\lceil \frac{819675 \cdot 522045}{11} \right\rceil = \lceil 38900657761.4 \dots \rceil = 38900657762.$$
		
		\textbf{Dimension $n=12$:} Using $a_{11}(P^{12}) \ge 33431059633$ and $v_{fin}(P^{11}) \ge 38900657762$:
		$$v_{fin}(P^{12}) \ge \left\lceil \frac{33431059633 \cdot 38900657762}{12} \right\rceil = 108374184117028860113.$$
		
		These values are summarized in Table 2.
	\end{proof}
	
	\begin{table}[h]
		\centering
		\renewcommand{\arraystretch}{1.4}
		\begin{tabular}{|c|c|c|}
			\hline
			Dimension $n$ & $\mathbf{v_{fin}(P^n) \ge}$ \\
			\hline
			7 & \textbf{4} \\
			8 & \textbf{41} \\
			9 & \textbf{1 312} \\
			10 & \textbf{522 045} \\
			11 & \textbf{38 900 657 762} \\
			12 & \textbf{108 374 184 117 028 860 113} \\
			\hline
		\end{tabular}
		\caption{New lower bounds on the number of finite vertices $v_{fin}(P^n)$.}
		\label{tab:comparison_fin}
	\end{table}
	
	\section*{Acknowledgments}
	The author expresses gratitude to A.\,Yu.~Vesnin for constant attention to the work and fruitful discussions.

\end{document}